\documentclass[reqno,a4paper,10pt]{amsart}
\usepackage{amsmath,amsfonts,amssymb,amscd,amsthm,amsbsy,upref,mathrsfs}
\usepackage[T1]{fontenc}

\newtheorem{theorem}{Theorem}[section]
\newtheorem{lemma}[theorem]{Lemma}
\newtheorem{proposition}[theorem]{Proposition}
\newtheorem{corollary}[theorem]{Corollary}

\newtheorem{definition}[theorem]{Definition}
\newtheorem*{notation}{Notation}

\theoremstyle{remark}
\newtheorem{remark}[theorem]{Remark}

\newcommand{\N}{{\mathbb{N}}}
\newcommand{\R}{{\mathbb{R}}}

\newcommand{\al}{\alpha}
\newcommand{\de}{\delta}
\newcommand{\mc}{\mathcal}
\newcommand{\mt}{\mc{T}}

\newcommand{\mrr}{\mc{R}}

\newcommand{\dd}{{\bf d}}
\newcommand{\mm}{{\bf m}}

\numberwithin{subsection}{section}
\numberwithin{equation}{section}

\DeclareMathOperator{\supp}{supp}
\DeclareMathOperator{\weight}{w}

\DeclareMathOperator{\ran}{range}

\DeclareMathOperator{\suc}{succ}

\begin{document}

\title{Small operator ideals on the Schlumprecht and Schreier spaces}
\author{Antonis Manoussakis} 
\address{School of Environmental Engineering, Technical University of Crete, 
University Campus, 73100  Chania, Greece}
\email{amanousakis@isc.tuc.gr}
\author{Anna Pelczar-Barwacz} 
\address{Institute of Mathematics, Faculty of Mathematics and Computer Science, Jagiellonian University, {\L}ojasiewicza 6, 30-348 Krak\'ow,  Poland}
\email{anna.pelczar@uj.edu.pl}

\maketitle

\begin{abstract}

We present a criterion for operators on a Banach space $X$ to generate distinct operator ideals in the algebra $\mathscr{B}(X)$ of bounded linear operators on $X$. 
We show  that there are exactly $2^\mathfrak{c}$ distinct small closed operator ideals on the Schlumprecht space and  on any Schreier space of finite order. 
\end{abstract}

\section{Introduction}

The algebras of bounded linear operators on a Banach spaces provide natural examples of non-commutative Banach algebras; 
the structure of the lattice of closed operator ideals is widely studied. We recall here only some of known results, for the thorough survey referring to \cite{bkl,ll,jps,sz}.   
The study of the closed operator ideals dates back to \cite{c}, where it was proved that the ideal of compact operators $\mathscr{K}(\ell_2)$ is the only non-trivial closed operator ideal on the Hilbert space $\ell_2$.   
The result was generalized by I.~Gohberg, A.~Markus and I.~Feldman to the case of $c_0$ and $\ell_p$, $1\leq p<\infty$. The list of Banach spaces with fully described lattice of closed operator ideals includes also non-separable Hilbert spaces \cite{g,l}, spaces $(\bigoplus_{n=1}^\infty\ell_2)_{c_0}$ and 
$(\bigoplus_{n=1}^\infty\ell_2)_{\ell_1}$  \cite{llr,lsz}, and - of completely different character - the celebrated Argyros-Haydon space $X_{AH}$ with the property $\mathscr{B}(X_{AH})/\mathscr{K}(X_{AH})\cong \R$, as well as some of spaces built on its basis, see \cite{kl,mpz,t1,t2}. 

Recently optimal results on the cardinality of the lattice of closed operator ideals were obtained for various classical spaces, answering long open questions of A.~Pietsch \cite{p}. Recall here that the maximal cardinality of a family of distinct closed ideals on a separable Banach space is $2^\mathfrak{c}$, whereas the maximal length of a chain of closed operator ideals is $\mathfrak{c}$.
The research was often based on the analysis of complemented subspaces of the considered space or of the family of strictly singular operators acting on the space. Recall here that an operator $T : X\to X$ is called \textit{strictly singular} if none of its restriction to an infinite-dimensional subspace of $X$ is an isomorphism onto its image. The family $\mathscr{S}(X)$ of bounded strictly singular operators on a Banach space $X$ forms a classical example of a closed operator ideal. It is useful to distinguish operator ideals with respect to the ideal of strictly singular operators; following \cite{jps} an operator ideal in $\mathscr{B}(X)$ is called \textit{small} if it is contained in $\mathscr{S}(X)$, otherwise it is called \textit{large}. 
We mention here the latest known results. 
W.B.~Johnson and G.~Schechtman proved in \cite{js} that $\mathscr{B}(L_p)$, $1<p<\infty$, $p\neq 2$, contains $2^\mathfrak{c}$ many distinct large and $2^\mathfrak{c}$ small closed operator ideals, whereas in \cite{jps} chains of continuum many small closed  ideals on the spaces $L_1(0,1)$, $C[0,1]$ and $L_\infty(0,1)$ were built. In the paper \cite{fsz2} the authors present a general method of constructing $2^\mathfrak{c}$ many distinct small operator ideals, and applied it  in particular in the case of the spaces $\ell_p\oplus\ell_q$, $1\leq p<q\leq\infty$, and $\ell_p\oplus c_0$, $1\leq p<\infty$, as well as in the case of products of reflexive $\ell_p$ spaces and convexified Tsirelson space and their duals. As a consequence another proof of the result of \cite{js} was obtained. In the remaining case of the spaces  $\ell_1\oplus c_0$ and $\ell_1\oplus \ell_\infty$  uncountable chains of closed ideals were built in \cite{sw}, this is the best known result concerning these spaces so far. 

Spaces less "classical" are also studied with respect to the lattice of closed ideals, as Figiel spaces \cite{ll,js}, the Tsirelson space and Schreier spaces of finite order \cite{bkl}. In the last paper the authors showed that in both cases there are continuum many maximal (large) ideals generated by projections on subspaces spanned by subsequences of the canonical bases of the considered spaces. In the case of the Tsirelson space the result can be strengthened, as this space admits $2^\mathfrak{c}$ many distinct closed operator operator ideals by W.B.~Johnson (cf. \cite{fsz2}). 
  
In the present paper we add to the list of Banach spaces with the lattice of closed operator ideals  of the maximal cardinality,  
the Schlumprecht space $X_S$, a fundamental example of Banach space theory, and the Schreier spaces of finite order $X[\mathcal{S}_N]$, $N\in\N$, in the second case solving Problem 16 in \cite{fsz2}. Both in the  case of the Schlumprecht space and Schreier spaces we build a family of small closed operator ideals $(I_A)_{A\subset\R}$ with the property that $A\subset B$ iff $I_A\subset I_B$ for any $A,B\subset\R$ (Theorems \ref{main} and \ref{main-schreier}).  

Recall that by \cite{w} the Schlumprecht space being complementably minimal \cite{as2} (i.e. any its closed infinite-dimensional subspace contain a complemented copy of the whole space) admits a unique maximal ideal, which is $\mathscr{S}(X_S)$, limiting the study on closed ideals on $X_S$ to the small ones. In the case of the Gowers-Maurey space $X_{GM}$, the first known hereditarily indecomposable space, defined on the basis of $X_S$, the ideal $\mathscr{S}(X_{GM})$ is also the unique maximal ideal, however, for different reasons. In \cite{as} strictly singular non-compact operators both on $X_S$ and $X_{GM}$ were constructed. Moreover it was noticed that $\ell_\infty/c_0$ can be embedded into $\mathscr{S}(X_{GM})/\mathscr{K}(X_{GM})$, however, to the best of our knowledge, it was not known how many, if any, there are closed ideals between $\mathscr{S}(X_{S})$ and $\mathscr{K}(X_S)$ (cf. \cite{ll}).
The proof of our result in the case of the Schlumprecht space is based on a refined construction of strictly singular operators of \cite{as}, with use of some techniques of \cite{mp2}, and on a general criterion we present in the second section (Prop. \ref{main-criterion}) concerning operators of a specific form with a certain $c_0$ type behaviour. Our approach in the case of Schreier spaces relies heavily on the characterization of domination of one subsequence of the canonical basis by another of \cite{gl} and uses the presence of $c_0$ inside the considered spaces. 

The above results leave the question on a Banach space with an unconditional basis admitting exactly $\mathfrak{c}$ many distinct operator ideals open (cf. \cite{fsz2}); the question in the general case of separable Banach spaces was answered in particular in \cite{mpz,t2}.

The paper is organized as follows: in the second section we present some criterion for generating distinct operator ideals by operators of a specific form, the third section is devoted to the case of the Schlumprecht space, the forth section concerns the Schreier spaces of finite order. 

\

We recall here the basic notions concerning trees, as we shall need them both in the case of the Schlumprecht space and the Schreier spaces. 
By a {\em  tree}  we shall mean a non-empty partially ordered  set $(\mt, \preceq)$ for which the set $\{ y \in \mt:y \preceq x \}$ is linearly ordered and finite for each $x \in \mt$.
The \textit{root} is the smallest element of the tree (if it exists). The {\em terminal}  nodes are the maximal elements. A {\em branch} of $\mt$ is any maximal linearly ordered set in $\mt$. A tree with no infinite branches is called \textit{well-founded}.  The {\em immediate successors}  of $x \in \mt$ are all  the nodes $y \in \mt$ such that $x \prec y$ but there is no $z \in \mt$ with $x \prec z \prec y$, a set of immediate successors of $x$ denote by $\suc (x)$. A tree is \textit{finitely branching}, if any its non-terminal element has finitely many immediate successors. 
If $\mt$ has a root, then for any node $\al\in\mt$ we define a level of $\al$, denoted by $|\al|$, to be the length of the branch linking $\al$ and the root. 

Throughout the paper we use the standard notation and terminology concerning Banach space theory following  \cite{lt}, see also \cite{ak}.

\section{A  criterion}
Given a Banach space $X$ by $\mathscr{B}(X)$ we denote the Banach algebra of bounded (linear) operators on $X$. 

We shall consider operator ideals in  $\mathscr{B}(X)$ for some Banach space $X$, in this setting a set $I\subset\mathscr{B}(X)$ is an operator ideal provided it contains finite-dimensional operators,  it is both a linear subspace and a two-sided ideal in the  algebra $\mathscr{B}(X)$.

For any family $\mathscr{B}\subset \mathscr{B}(X)$ by $I(\mathscr{B})$ denote the closed operator ideal generated by $\mathscr{B}$. Recall that $I(\mathscr{B})$  is the closure in the operator norm of the set $\{\sum_{j=1}^lQ_j\circ R_j\circ S_j: Q_j,S_j\in\mathscr{B}(X), R_j\in\mathscr{B}, \    j=1,\dots, l, \ l\in\N\}$.

\

Let $X$ be a Banach space. Let $(e_n)_n\subset X$ be a normalized basic sequence and let $Y$ be the closed subspace spanned by $(e_n)_n$. Denote by $(e_n^*)_n\subset Y^*$ the sequence of biorthogonal functionals to $(e_n)_n\subset Y$.

\begin{notation}
Let $\mathscr{B}_{seq}(X,Y)\subset \mathscr{B}(X)$ denote the family of bounded operators on $X$ of the form 
\begin{align}\label{op}
T: X\ni x\mapsto\sum_{n=1}^\infty f_n(x)e_n\in Y\subset X
\end{align}
for some biorthogonal  basic sequences $(f_n,x_n)_{n\in\N}\subset X^*\times X$ with $(f_n)_n\subset B_{X^*}$ weakly$^*$ null and $(x_n)_n$ seminormalized (in particular $Tx_n=e_n$ and $e^*_n\circ  T=f_n$, $n\in\N$).
\end{notation}

\begin{remark}
Obviously any bounded operator $T:X\to Y$ is of the form \eqref{op} with $(f_n)_n\subset X^*$ bounded, the important part of the definition above concerns a seminormalized sequence $(x_n)_n$ biorthogonal to $(f_n)_n$. Any operator $T\in\mathscr{B}_{seq}(X,Y)$ we shall consider always with fixed sequences $(f_n,x_n)_n\subset X^*\times X$ associated to $T$ as above. 
\end{remark}

For any  $T\in\mathscr{B}_{seq}(X,Y)$ with the associated sequences  $(f_n,x_n)_n$ let 
\begin{align*}
c_k(T)&:=\limsup_{i_1<\dots<i_k, i_1, \dots,i_k\to\infty}\sup\left\{\|\epsilon_1 f_{i_1}+\dots+\epsilon_k f_{i_k}\|:\epsilon_1=\pm 1, \dots, \epsilon_k=\pm 1\right\}, \\ d_k(T)&:=\limsup_{i_1<\dots<i_k, i_1,\dots,i_k\to\infty}\|x_{i_1}+\dots+x_{i_k}\|, \ \ \text{ for any } k\in\N
\end{align*}

\begin{notation}
Let $\mathscr{S}_\infty(X,Y)$ denote the family of bounded operators $T:X\to Y$ with the following property: 
for any bounded sequence $(y_m)_m\subset X$ with $\|Ty_m\|_{Y,\infty}\xrightarrow{m\to\infty}0$ also $\|Ty_m\|\xrightarrow{m\to\infty}0$, where the norm $\|\cdot\|_{Y,\infty}$ on $Y$ is given by
 $ \|\cdot\|_\infty=\sup\{|e_n^*(\cdot)|:n\in\N\}$.
\end{notation}

\begin{remark}
Notice that for $X$ not containing an isomorphic copy of $c_0$ we have $\mathscr{S}_\infty(X,Y)\subset\mathscr{S}(X)$. On the other hand, if the sequence $(e_n)_n$ is equivalent to the unit vector basis of $c_0$, then any operator of the form \eqref{op} with $(f_n)_n\subset B_{X^*}$ belongs to $\mathscr{B}_{seq}(X,Y)\cap \mathscr{S}_\infty(X,Y)$. 
\end{remark}
With the abuse of notation we use $Y$ instead of its basis $(e_n)_n$ both in the case of $\mathscr{B}_{seq}(X,Y)$ and $\mathscr{S}_\infty(X,Y)$, however, it will be clear which basis of $Y$ we consider.

\begin{proposition} \label{main-criterion}
 Fix an operator $T\in\mathscr{B}_{seq}(X,Y)$ defined by $(f_n,x_n)_n$ and a non-empty family $\mathscr{B}\subset\mathscr{B}_{seq}(X,Y)\cap \mathscr{S}_\infty(X,Y)$.  Assume 
 \begin{align}\label{eq-main-criterion}
  \inf_{k\in\N}\frac{c_k(R)d_k(T)}{k}=0 \ \ \ \text{ for any }R\in\mathscr{B}
 \end{align}
  Then $\|T-Q\|\geq 1/d$ for any $Q$ in the operator ideal $I(\mathscr{B})$ generated by $\mathscr{B}$, where $d:=\sup_n\|x_n\|$.
 \end{proposition}
\begin{remark}
It is not difficult to notice that operators $T,R\in\mathscr{B}_{seq}(X,Y)$ with \eqref{eq-main-criterion} satisfy $\|T-R\|\geq\frac{1}{d}$, dealing with operators in $\mathscr{S}_\infty(X,Y)$ allows separating not only operators, but also operators from operator ideals. 
\end{remark}

\begin{proof}[Proof of Prop. \ref{main-criterion}]
Take operators as in the theorem. 

Assume there is $Q\in I(\mathscr{B})$ with $\|T-Q\|<\frac{1}{d}$. Then there are $R_j\in\mathscr{B}$, $Q_j,S_j\in\mathscr{B}(X)$, $j=1,\dots,l$ such that
$\|T-\sum_{j=1}^lQ_j\circ R_j\circ S_j\|< \frac{1}{d}$. As $Tx_n=e_n$ and $(x_n)_n$ is bounded, $n\in \N$, passing to a subsequence if necessary, we find $j_0\in\{1,\dots,l\}$ such that the sequence $((Q_{j_0}\circ R_{j_0}\circ S_{j_0})x_n)_n$, is seminormalized. It follows that  $((R_{j_0}\circ S_{j_0})x_n)_n$ is seminormalized as well.

In order to simplify the notation write $R:=R_{j_0}$,  $S:=S_{j_0}$. Let $(g_n)_n\subset X^*$ be the seminormalized basic sequence defining $R$. As $R\in\mathscr{S}_\infty(X,Y)$ there is universal  $c>0$ with $\sup_{l\in\N}|e_l^*(R(Sx_n))|>c$ for all $n$. For any $n\in \N$ pick $l_n\in \N$ and $\epsilon_n=\pm 1$ so that $\epsilon_ng_{l_n}(Sx_n)=\epsilon_ne_{l_n}^*(R(Sx_n))\geq c$ for each $n$.

Notice that for any fixed $n\in\N$ we have $g_{l_n}(Sx_m)\xrightarrow{m\to\infty}0$ $(*)$. Indeed, assume $|g_{l_n}(Sx_m)|>\de$, for some universal $\delta$ and any $m\in M$, with some infinite $M\subset\N$. Without loss of generality we assume $g_{l_n}(Sx_m)>\de$, $m\in M$. Thus for any $k\in\N$ and any $A\subset M$ with $\# A=k$ we have 
$$
\|S\|\|\sum_{m\in A}x_m\|\geq \|S(\sum_{m\in A}x_m)\|\geq g_{l_n}(\sum_{m\in A}Sx_m)\geq k \de
$$ 
Therefore $\liminf_{k\in\N}\frac{d_k(T)}{k}\geq \delta\|S\|^{-1}>0$, which contradicts the assumption on $(d_k(T))_k$ (as $\liminf_kc_k(R)>0$).

It follows that we can pass to subsequence of $(x_n)_n$ on which the mapping $n\mapsto l_n$ is an injection, thus 
$g_{l_n}(Sx_m)\xrightarrow{n\to\infty}0$ for any $m\in\N$ $(**)$ (recall $R\in\mathscr{B}_{seq}(X,Y)$, thus $(g_{l_n})_n$ is weakly$^*$ null).

Using $(*)$ and $(**)$ we easily pass to a further subsequence $(x_n)_{n\in L}$ such that $|g_{l_n}(Sx_m)|\leq \frac{c}{2^{n+m}}$ for any $n\neq m$, $n,m\in L$.
Now for any $k\in\N$ choose $A\subset L$, $\# A= k$, with $\min A>\frac{2}{c}$ big enough to guarantee $\|\sum_{n\in A}\pm g_{l_n}\|\leq 2c_k(R)$ and $\|\sum_{n\in A}x_n\|\leq 2d_k(T)$. Estimate 
\begin{align*}
\|S\|2d_k(T)&\geq\|S\|\|\sum_{n\in A}x_n\|
\geq\|\sum_{n\in A}Sx_n\|\\
&\geq\frac{1}{2c_k(R)}\sum_{n\in A}\epsilon_n g_{l_n}\left(\sum_{n\in A}Sx_n\right)\\
& = \frac{1}{2c_k(R)}\sum_{n\in A}\epsilon_n g_{l_n}(Sx_n) + \frac{1}{2c_k(R)}\sum_{n\neq m\in A}\epsilon_n g_{l_n}(Sx_m)\\
&\geq c \cdot\frac{k}{2c_k(R)}-\frac{1}{2c_k(R)}\cdot\frac{c}{2}\geq \frac{c}{4} \cdot\frac{k}{c_k(R)}
\end{align*}
Thus for any $k\in\N$
$$
\frac{c_k(R)d_k(T)}{k}\geq \frac{c}{8\|S\|}
$$
which contradicts \eqref{eq-main-criterion} and thus ends the proof of the theorem.
\end{proof}

\section{Schlumprecht space}

\subsection{Basic definitions and properties}

We recall  the definition of the Schlumprecht space $X_S$ \cite{sch}. Let $K\subset c_{00}$ be the smallest set containing unit vectors $\{\pm e_n^*:n\in\N\}$ and such that for any $n\leq m$ and any block sequence $f_1<\dots<f_n$ of elements of $K$ also the weighted average $\frac{1}{\log_2(m+1)}(f_1+\dots+f_n)$ belongs to $K$. 
The set $K$ defines a norm $\|\cdot\|$ on $c_{00}$ as its norming set, i.e. $\|\cdot\|=\sup_{f\in K}|f(\cdot)|$, where for any $f=(a_n)_n\in K$ and $x=(b_n)_n\in c_{00}$ we have $f(x)=\sum_na_nb_n$.
The Schlumprecht space $X_S$ is defined as the completion of $(c_{00},\|\cdot\|)$.

\begin{remark}\label{schlum}
\begin{enumerate}
    \item It follows straightforward from the definition of the Schlum\-precht space that the unit vector basis $(e_n)_n$ is 1-unconditional and 1-subsymmetric (i.e. 1-equivalent to any of its subsequences). 
\item By definition of the norming set $K$ any $f\in K\setminus\{\pm e^*_n:n\in\N\}$ has a tree-analysis $(f_\gamma)_{\gamma\in\mathcal{R}}\subset K$ where $\mathcal{R}$ is a well-founded finitely-branching tree with the root $\emptyset$, $f_\emptyset=f$, $f_\gamma=\frac{1}{\log_2(m_\gamma+1)}\sum_{\zeta\in\suc(\gamma)}f_\zeta$, $m_\gamma\geq \#\suc(\gamma)$, for any non-terminal $\gamma\in\mathcal{R}$, $f_\gamma=\pm e_{n_\gamma}$ for any terminal $\gamma\in\mathcal{R}$. In such situation a weight $\weight(f)$  of $f$ is given by $\weight(f)=m_\emptyset$.

\item \cite{sch} We have $\|e_1+\dots+e_k\|=\frac{\log_2(k+1)}{k}$ for any $k\in\N$. Thus for any $h\in K$ we have $\#\{i\in\supp h: |h(e_i)|\geq \log_2(k+1)^{-1}\}\leq k$ for any $k\in\N$.
\end{enumerate}
\end{remark}

We recall now a notion from \cite{mp2}, implicitly used in \cite{as}. 

\begin{definition} A core tree is any finitely branching tree $\mt\subset\cup_n\N^n$  with no terminal nodes and a root $\emptyset$, considered with the tree order $\preceq$ and the lexicographic order $\leq_{lex}$. For any $n\in\N$ let $\mt_n=\{\al\in\mt: |\al|\leq n\}$. For any $\al\in\mt$ let $m_\al=\#\suc (\al)$.
\end{definition}

\begin{notation}

We enumerate nodes of $\mt$ according to the lexicographic order as $(\al_j)_{j\in\N}$, starting from $\al_0=\emptyset$. For simplicity  we write $m_j=m_{\al_j}$,  $j\in\N$.

For any $\al\in\mt$, $\al\neq\emptyset$,  let $c_\al=\prod_{\gamma\prec\al}\log_2(m_\gamma+1)^{-1}$ and $d_\al=\prod_{\gamma\prec\al}\log_2(m_\gamma+1)m_\gamma^{-1}$. For simplicity we  write $c_j=c_{\al_j}$, $d_j=d_{\al_j}$, $j\in\N$.

\end{notation}

\begin{remark} Notice that any sequence of parameters $(m_j)_{j\in\N}\subset\N\setminus\{0\}$ defines a unique core tree $\mt$, for which  $\#\suc(\al_j)=m_j$ for any $j\in\N$, with the enumeration $(\al_j)_{j\in\N}$ of elements of $\mt$ according to the order $\leq_{lex}$.
\end{remark}

The definition below recalls the notion of repeated averages of \cite{as}.

\begin{definition}\label{RIS-tree-an}\cite{as}
We say that a vector $x\in X_S$ has a tree-analysis $(x_\al)_{\al\in\mt_n}\subset X_S$ (of height $n\in\N$) with a core tree $\mt$ with associated parameters $(m_\al)_{\al\in\mt}\subset\N$, if the following hold:
\begin{enumerate}
\item $x_\emptyset=x$,
\item for any terminal node $\al\in\mt_n$ we have  $x_\al=e_{t_\al}$ for some $t_\al\in \N$,
\item for any non-terminal node $\al\in\mt$ the vector $x_\al$ is a seminormalized $m_\al$-average of
$(x_\beta)_{\beta\in\suc(\al)}$ of the form $x_{\al}=\frac{\log_2(m_\al+1)}{m_\al}\sum_{\beta\in\suc(\al)}x_\beta$, 
\item for any nodes $\al,\beta\in\mt$ with $|\al|=|\beta|$, $\al\leq_{lex}\beta$ we have $x_\al<x_\beta$. 
\end{enumerate}
\end{definition}

\begin{lemma}\cite[Prop. 2.3, Lemma 2.10]{mp2}\label{lemma-RIS-tree}
 Let $\mt$ be a core tree with the associated sequence of parameters $(m_j)_{j\in\N}$. If for some $(q_j)_{j\in\N}\subset\N$ the sequence $0<m_0<q_0<m_1<q_1<\dots$ increases fast enough, then for any tree-analysis $(x_\al)_{\al\in\mt_n}$, $n\in\N$, of a vector $x\in X_S$ we have
 \begin{enumerate}
 \item[(F1)] $1\leq\|x_\al\|\leq 2$ for any $\al\in\mt_n$,
\item[(F2)] $|h(x_{\al_j})|\leq \frac{2}{\log_2(m_j+1)}$ for any $h\in K$ with $\weight(h)\geq q_j$ and any $\al_j\in\mt_n$.
 \end{enumerate}
\end{lemma}

\begin{remark}\label{remark-RIS-tree}
The precise conditions guaranteeing the fast increase of $(m_j,q_j)_j$ are given in \cite{mp2}. The important feature of this notion is that it is "finitely defined" in a sense that if a sequence $m_0<q_0<\dots<m_s<q_s$ can be extended to fast increasing sequence (required in Lemma \ref{lemma-RIS-tree}), then any sequence $m_0<q_0<\dots<m_s<q $, with $q>q_s$, as well, and the same holds true for sequences of the form $m_0<q_0<\dots<q_{s-1}<m_s$.
\end{remark}

\begin{remark}\label{remark-l1-averages}
In the situation as above it follows that any $x_\al$ is an $\ell_1^r$-average with constant 3, with $r=r(m_\al)\to \infty$ as $m_\al\to\infty$. Thus by standard reasoning for any $\al\in\mt$ there is some $q=q(m_\al)$  such that for any $E_1<\dots<E_q$, we have $\sum_{s=1}^q\|E_sx_\al\|\leq 4$, with $q(m_\al)\to\infty$ as $m_\al\to\infty$ (see
\cite{sch}).
\end{remark}

\begin{definition}
Given a vector  $x$ with a tree-analysis $(x_{\al})_{\al\in \mt_n}$ as in Def. \ref{RIS-tree-an} we define in a natural way \textit{an associated norming functional}
 $f\in K$ with a tree-analysis $(f_\al)_{\al\in\mt_n}$ as follows (with the notation of Def. \ref{RIS-tree-an}):
\begin{enumerate}
\item $f_\emptyset=f$,
\item for any terminal node $\al\in\mt$ we set $f_\al=e_{t_\al}^*$,
\item for any non-terminal node $\al\in\mt$ we set  $f_{\al}=\frac{1}{\log_2(m_\al+1)}\sum_{\beta\in\suc(\al)}f_\beta$.
\end{enumerate}
\end{definition}
In the situation above obviously any $f_\al$, $\al\in\mt_n$, of the above form is in the norming set of $X_S$, thus in particular $f\in K$, and $f(x)=1$.

\subsection{Block sequences of vectors and functionals defined by a core tree}

For the rest of this section we fix a core tree $\mt$ with parameters $(m_j)_j$, a block sequences $(x_n)_n\subset X_S$ and $(f_n)_n\subset X_S^*$ defined by $\mt$, i.e. such that each  $x_n$ has a tree-analysis $(x_\al^n)_{\al\in\mt_n}$ and each $f_n$ is an associated functional to $x_n$ with a tree-analysis $(f^n_\al)_{\al\in\mt_n}$. We shall estimate in this section finite sums of $(x_n)_n$ and $(f_n)_n$ under certain conditions on parameters associated to the tree $\mt$.

For a sequence of parameters $(m_j,q_j)_j$ with $0<m_0<q_0<m_1<q_1<\dots$ we define the following conditions of the fast growth: for any tree-analysis $(x_\al)_{\al\in\mt_n}$, $n\in\N$, of a vector $x\in X_S$ we require
\begin{enumerate}
  \item[(F3)] $\sum_{s=1}^{q_{j-1}}\|E_sx^n_{\al_j}\|\leq 4$ for any $E_1<\dots<E_{q_{j-1}}$,  any $j\in\N$ and $n\geq |\al_j|$,
 \item[(F4)] $\log_2(m_j+1)\geq 2\log_2(m_{j-1}+1)$ for all $j\in\N$,
\item[(F5)] $\# \{\al: |\al|=N\}(=\sum_{\al: |\al|=N-1}m_\al)\leq \log_2(q_{s_N}+1)$, where $\al_{s_N}$ is maximal in $\{\al: |\al|=N-1\}$ with respect to $\leq_{lex}$,
\end{enumerate}

\begin{notation}
 As in the case of $(m_\al)_{\al\in\mt}$ we  write  $q_{\al_j}=q_j$ for any  $j\in\N$.
\end{notation}

\begin{remark}\label{rem-fast-growth} Notice that it is possible to construct inductively $(m_j,q_j)_j$ satisfying (F1)-(F5). Indeed, at each step choose $q_j$, $m_j$ big enough to ensure conditions (F1) and (F2) are satisfied (cf. Remark \ref{remark-RIS-tree}), and moreover, given $m_0<q_0<\dots<m_{j-1}<q_{j-1}$ choose $m_j$ big enough so that (F3) is satisfied (according to Remark \ref{remark-l1-averages}) as well as (F4). Additionally, given $m_0<q_0<\dots<m_j$, with $j=s_N$ for some $N\in\N$, choose $q_j$ so that (F5) is satisfied.  Moreover, the choice of $(m_j,q_j)_j$ is "finitely defined" in the sense of Remark \ref{remark-RIS-tree}.
\end{remark}

\begin{proposition}\label{est-vectors}
 Let $(x_n)_n\subset X_S$ be a block sequence defined by a core tree $\mt$ with $(m_j,q_j)_j$ satisfying (F1)-(F5). Fix $j_0\in\N$ so that $\al_{j_0}$ is minimal in $\{\al\in\mt: |\al|=|\al_{j_0}|\}$ with respect to $\leq_{lex}$. Then for any $q_{j_0-1}\leq k\leq m_{j_0}$ and $j_0<n_1<\dots<n_k$ we have
 $$
\frac{1}{2} kd_{j_0}\leq \|x_{n_1}+\dots +x_{n_k}\|\leq 14k d_{j_0}
$$
\end{proposition}

\begin{proof}

As the bases of $X_S$ and $X_S^*$ are 1-unconditional, and the vectors $(x^n_\al)$ have non-negative coefficients, we can assume that the functionals analysed below have also non-negative coefficients.

Let $N:=|\al_{j_0}|$.  Fix $k\in\N$ and $n_1,\dots,n_k$ as in the proposition. Let each $x_{n_i}$ has the tree-analysis $(x_\al^i)_{\al\in\mt_{n_i}}$. 

For the left estimate consider the associated functionals $(f_{n_i})_i$ with corresponding tree-analysis $(f^i_\al)_{\al\in\mt_{n_i}}$, $i=1,\dots,k$,  and estimate using unconditionality and (F4) as follows
\begin{align*}
 \|x_{n_1}+\dots+x_{n_k}\|&\geq d_{j_0}\|x^1_{\al_{j_0}}+\dots+x^k_{\al_{j_0}}\|\\
 & \geq d_{j_0} \frac{1}{\log_2(k m_{j_0}+1)}\sum_{i=1}^k\sum_{\beta\in\suc(\al_{j_0})}f_\beta^i(x^i_{\al_{j_0}})\\
 & \geq \frac{1}{2} d_{j_0}\sum_{i=1}^k \frac{1}{\log_2(m_{j_0}+1)}\sum_{\beta\in\suc(\al_{j_0})}f_\beta^i(x^i_{\al_{j_0}})\\
 &= \frac{1}{2} d_{j_0}\sum_{i=1}^k f^i_{\al_{j_0}}(x^i_{\al_{j_0}})\\
 &= \frac{1}{2} kd_{j_0}.
\end{align*}

For the right estimate we show first that
\begin{align}\label{est1}
 \|x_{n_1}+\dots +x_{n_k}\|\leq 4k d_{j_0} + 5\|y_1+\dots+y_k\|
\end{align}
 where $(y_1,\dots,y_k)$ is a block sequence of shifted copies of the vector $x_{N}=\sum_{|\al_j|=N}d_{\al_j}e_{t_j}$ and for such $(y_1,\dots,y_k)$ we prove
 \begin{align}\label{est2}
  \|y_1+\dots+y_k\|\leq 2k d_{j_0}
 \end{align}

For \eqref{est1} take a functional $\hat{h}\in K$ with its tree-analysis $(\hat{h}_\gamma)_{\gamma\in\hat{\mathcal{R}}}$. By \cite{mp} we can pick $h\in K$ with $|\hat{h}(\sum_i\sum_{|\al|=N} x_\al^i)|\leq 6h(\sum_i\sum_{|\al|=N} x_\al^i)$ and a tree-analysis $(h_\gamma)_{\gamma\in\mathcal{R}}$ compatible with the block sequence $(x^i_{\al}: |\al|=N, i=1,\dots, k)$ (meaning that for any $\gamma\in\mathcal{R}$, $\al\in \mt$ with $|\al|=N$ and $i=1,\dots,k$ we have either $\ran x_\al^i\cap\ran h\subset\ran h_\gamma$, $\ran x_\al^i\cap\ran h\supset\ran h_\gamma$ or $\ran x_\al^i\cap\ran h_\gamma=\emptyset$). We want to estimate
$$
h(x_{n_1}+\dots+x_{n_k})=\sum_{\al\in\mt, |\al|=N}d_\al h(x^1_\al+\dots+x^k_\al)
$$

In the sequel we consider only those $(i,j)$, $|\al_j|=N$, for which $h(x^i_{\al_j})\neq 0$. For any $(i,j)$, $|\al_j|=N$, let $\gamma_{i,j}\in\mathcal{R}$ be the node with $h_{\gamma_{i,j}}$ covering $x^i_{\al_j}$ (i.e. $\gamma_{i,j}$ is maximal in $\mathcal{R}$ with $\ran h_\gamma\supset \ran h\cap \ran x_{\al_j}^i$).  

Let $A=\{(i,j): i=1,\dots,k, |\al_j|=N, \weight(h_{\gamma_{i,j}})\leq q_{j-1}\}$. Then by (F3)
$$
|\sum_{\beta\in\suc(\gamma_{i,j})}h_\beta(x^i_{\al_j})|\leq 4, \text{ for each } (i,j)\in A
$$
thus we can replace in the tree-analysis of $h$ all successors of $h_{\gamma_{i,j}}$ with range intersecting the range of $x^i_{\al_j}$ by one functional $f^i_{\al_j}$, paying the cost of multiplying the action of $h$ on $x_{n_1}+\dots+x_{n_k}$ by 4. 
Thus
\begin{align}
 |h(\sum_{(i,j)\in A}d_j x^i_{\al_j})|\leq 4 \|y_1+\dots+y_k\|
\end{align}

Let $B=\{(i,j): i=1,\dots,k, |\al_j|=N, \weight(h_{\gamma_{i,j}})\geq q_j\}$. Then for any $(i,j)\in B$ by (F2)  we obtain
$$
|h_{\gamma_{i,j}}(x^i_{\al_j})|\leq \frac{2}{\log_2(m_j+1)}
$$
Thus by (F4)
\begin{align}
 |h(\sum_{(i,j)\in B}d_j x^i_{\al_j})|\leq 2k\sum_{|\al_j|=N}\frac{d_{j_0}}{\log_2(m_j+1)}\leq 2kd_{j_0}
\end{align}

We consider one more set which can have non-empty intersection with the previous ones.

Let $C=\{(i,j): i=1,\dots,k, |\al_j|=N, \weight (h_\gamma)\geq q_{j_0-1} \text{ for some }\gamma\prec\gamma_{i,j}\}$.
Then for each $(i,j)\in C$ we have $|h(x^i_{\al_j})|\leq \log_2(q_{j_0-1}+1)^{-1}$, thus
\begin{align}
 |h(\sum_{(i,j)\in C}d_j x^i_{\al_j})|\leq kd_{j_0} \frac{\# \{j: |\al_j|=N\}}{\log_2(q_{j_0-1}+1)}\leq kd_{j_0}
\end{align}
with the last inequality by (F5) (as $j_0-1=s_N$ in the notation of (F5)).

Let $E=\{(i,j): i=1,\dots,k, |\al_j|=N\}\setminus (A\cup B\cup C)$. Then for any $(i,j)\in E$ we have $q_{j-1}<\weight(h_{\gamma_{i,j}})<q_j$. Notice that for $(i,j)\in E$ and $(i',j')$ with $h_{\gamma_{i,j}}(x^{i'}_{\al_{j'}})\neq 0$, we have $\gamma_{i,j}\preceq \gamma_{i',j'}$ (recall the tree-analysis of $h$ is compatible with $(x^i_{\al_j})_{i=1,\dots,k,|\al_j|=N}$), thus either $h_{\gamma_{i,j}}=h_{\gamma_{i',j'}}$ (and thus $j=j'$) or $(i',j')\in C$.

We split $E$ into two pieces: $E'=\{(i,j)\in E: h_{\gamma_{i,j}}(x^{i'}_{\al_{j'}})=0\ \  \forall (i',j')\in E, (i',j')\neq(i,j)\}$ and $E''=E\setminus E'$. By the remark above we  have $\# E''\leq k$ and thus
\begin{align}
 |h(\sum_{(i,j)\in E''}d_{\al_j}x^i_{\al_j})|\leq \# E'' d_{\al_{j_0}}\leq kd_{j_0}
\end{align}

By the definition of $E'$ we can replace in the tree-analysis of $h$ each $h_{\gamma_{i,j}}$, $(i,j)\in E'$, by the functional $f^i_{\al_j}$, not decreasing the action of the resulting functional on $\sum_{(i,j)\in E'}d_{\al_j}x^i_{\al_j}$.
Thus
\begin{align}
 |h(\sum_{(i,j)\in E'}d_{\al_j}x^i_{\al_j})|\leq \|y_1+\dots+y_k\|
\end{align}
which ends the proof of \eqref{est1}.

For \eqref{est2} take any $h\in K$ and let $F=\{i\in\N: |h(e_i)|>\log_2(q_{j_0-1}+1)^{-1}\}$. Then by Remark \ref{schlum}(3) and the choice of $k\geq q_{j_0-1}$ we have
\begin{align}
|h|_F(y_1+\dots+y_k)|\leq \# F \cdot d_{j_0}   \leq q_{j_0-1}d_{j_0}\leq  kd_{j_0}
\end{align}
On the other hand estimate, using (F5)
\begin{align*}
|h|_{\N\setminus F}(y_1+\dots+y_k)|&\leq \#\supp (y_1+\dots+y_k)\|y_1+\dots+y_k\|_\infty\|h|_{\N\setminus F}\|_\infty\\
&\leq k(\#\supp x_N) d_{j_0} \frac{1}{\log_2(q_{j_0-1}+1) }\\
&=k\#\{\al\in\mt: |\al|=N\} d_{j_0}\frac{1}{\log_2(q_{j_0-1}+1) }\\
&\leq kd_{j_0}
\end{align*}
which ends the proof of \eqref{est2}.
\end{proof}

For any $j_0\in\N$ let
 $$
 A_{j_0}=\{\al\in\mt: |\al|=|\al_{j_0}|, \al\geq_{lex}\al_{j_0}\}\cup \bigcup_{|\al|=|\al_{j_0}|, \al<_{lex}\al_{j_0}}\suc\al, \ \ \ A_{j_0}^\prime=A_{j_0}\setminus \{\al_{j_0}\}
 $$
\begin{remark}\label{rem-est-functionals}
 $\# A_{j_0}\leq \sum_{j <j_0}m_j$  and $m_\al\geq m_{j_0+1}$ for each $\al\in A_{\al_{j_0}}^\prime$.
\end{remark}

\begin{proposition}\label{est-functionals}
 Let $(f_n)_n\subset X_S^*$ be a block sequence of functionals given by a core tree $\mt$ with parameters $(m_j)_j$. Fix $\al_{j_0}\in\mt$. Then for any $k\leq m_{j_0}$ and ${j_0}+1<n_1<\dots<n_k$ we have
  $$
 \|f_{n_1}+\dots+f_{n_k}\|\leq 2\sum_{j <j_0}m_j
 $$
 and for any $k\leq m_{j_0+1}$ and $|\al_{j_0}|+1<n_1<\dots<n_k$ we have
 $$
 \|\sum_{i=1}^k f_{n_i}-c_{j_0}\sum_{i=1}^k f_{\al_{j_0}}^i\|\leq 2\sum_{j <j_0}m_j
 $$
 \end{proposition}
\begin{proof}
Let any $f_{n_i}$ have the tree-analysis $(f_\al^i)_{\al\in\mt_{n_i}}$. Then $f_{n_i}=\sum_{\al\in A^\prime_{\al_{j_0}}}c_\al f^i_\al+c_{j_0}f^i_{\al_{j_0}}$ for each $i=1,\dots,k$, where each $\al\in A^\prime_{\al_{j_0}}$ is not terminal in $\mt_{n_i}$ as $n_i>j_0+1\geq |\al_{j_0}|+1$ for any $i=1,\dots,k$. We show the second estimate, as the proof of the first estimate is a simplified version of the proof of the second one. Estimate, using second part of Remark \ref{rem-est-functionals}, as follows
\begin{align*}
\|\sum_{i=1}^k f_{n_i}-c_{j_0}\sum_{i=1}^k f_{\al_{j_0}}^i\|&\leq \|\sum_{i=1}^k\sum_{\al\in A^\prime_{\al_{j_0}}}\frac{1}{\log_2(m_\al+1)}\sum_{\beta\in\suc\al} f^i_\beta \|\\
&\leq \sum_{\al\in A^\prime_{\al_{j_0}}}\|\frac{1}{\log_2(m_\al+1)}\sum_{i=1}^k\sum_{\beta\in\suc\al}f^i_\beta\|\\
&\leq \sum_{\al\in A^\prime_{\al_{j_0}}}\|\frac{2}{\log_2(km_\al+1)}\sum_{i=1}^k\sum_{\beta\in\suc\al}f^i_\beta\|\\
&\leq 2\# A_{\al_{j_0}}
\end{align*}
as each functional $\frac{1}{\log_2(km_\al+1)}\sum_{i=1}^k\sum_{\beta\in\suc\al}f^i_\beta$ is in the norming set $K$ of $X_S$.
\end{proof}

\subsection{Operator defined by a core tree}

We fix a core tree $\mt$ with parameters $(m_j)_j$. The aim of this section is to study properties of an operator $T\in\mathscr{B}_{seq}(X_S):=\mathscr{B}_{seq}(X_S,X_S)$ (where $X_S$ is considered with the canonical basis $(e_n)_n$), defined by a sequence $(f_n,x_n)_n\subset X_S^*\times X_S$ associated to $\mt$. 

\begin{remark}
 In \cite{mp2} the index set of summation was required to be more lacunary, however, in the case of the Schlumprecht space we profit from the  form of involved coefficients $(\frac{1}{\log_2(n+1)})_n$ and dealing with the families $\mathcal{A}_n$.
\end{remark}

 For an increasing sequence $(m_j)_j\subset\N$ we define the following properties:
\begin{enumerate}
\item[(F6)] $2^j\sum_{t<j}m_t\leq  \log_2(m_j+1)$, $j\in\N$,
 \item[(F7)] $c_j\leq 1/2^{j+1}$, $j\in\N$.
 \end{enumerate}

We recall now estimate used already in  \cite{mp2}.
Any sequence $(f_n)_{n\in \N}$ defines a family of norms $\|\cdot\|_j$, $j\in\N$, on $X_S$ in the following way:
$$
\|\cdot\|_j=\sup\{|f_{n_1}(\cdot)|+\dots+|f_{n_j}(\cdot)|: n_1<\dots<n_j\}
$$

\begin{proposition}\label{est-functionals-2}
Take  a block sequence of functionals $(f_n)\subset X_S^*$  associated to a core tree $\mt$ with $(m_j)_j$ satisfying (F6)-(F7). Then the operator $T: X_S\ni x\mapsto \sum_n f_n(x)e_n\in X_S$ is well-defined and satisfies the following 
$$
\|Ty\|\leq \sum_{j=1}^{j_0}\|y\|_{m_j}+\frac{7}{2^{j_0}}\|y\|,  \ \  y\in X_S, \ \ \ j_0\in\N
$$
\end{proposition}

\begin{proof}
The calculation below proves both statements ($T$ well-defined and the estimate). The calculation for $j_0=1$ and $y\in X_S$ with finite support show that for any $y$ finite support we have $\|Ty\|\leq 8\|y\|$, proving that $T$ is well-defined on the whole space $X_S$.

Put $\epsilon_j=\log_2(m_j+1)^{-1}$, $j\in\N$, with $\epsilon_0=1$.
Let each $f_n$, has the tree-analysis $(f_\al^n)_{\al\in\mt_n}$.
Let $h\in K$. For any $j\in\N$ let
\begin{align*}
B_j&=\{ n\in \N:\ \epsilon_{j+1}<|h(e_n)|\leq \epsilon_j\}, \ \ \   D_j=B_j\cap\{1,\dots,j+2\},
\end{align*}
Notice that $\# B_j\leq m_{j+1}$ for any $j\in\N$ by Remark \ref{schlum}(3). Therefore by Prop. \ref{est-functionals} for any $j\in\N$ we have
\begin{align}\label{est-functionals-2-a}
\|\sum_{n\in B_j\setminus D_j}\pm (f_n-c_jf_{\al_j}^n)\|\leq 2\sum_{t<j}m_t
\end{align}
Now for a fixed $j_0\in\N$ and $y\in X_S$ we have
\begin{align*}
|h(Ty)|&=|h(\sum_{n\in \N}f_n(y)e_n)|\\
&\leq \sum_{j=0}^{j_0-1}|\sum_{n\in B_j}f_n(y)h(e_n)|+|h(\sum_{j=j_0}^\infty\sum_{n\in B_j}f_n(y)e_n)|\\
&\leq \sum_{j=0}^{j_0-1}\|y\|_{m_{j+1}}+\sum_{j=j_0}^\infty|\sum_{n\in B_j}(f_n-c_jf_{\al_j}^n)(y)h(e_n)|+\|\sum_{j=j_0}^\infty\sum_{n\in B_j} c_jf^n_{\al_j}(y)e_n\|
\end{align*}
Estimate the second term as follows, using \eqref{est-functionals-2-a} and (F6):
\begin{align*}
 \sum_{j=j_0}^\infty|\sum_{n\in B_j}(f_n-c_jf_{\al_j}^n)(y)h(e_n)|&\leq \sum_{j=j_0}^\infty|\sum_{n\in D_j}(f_n-c_jf_{\al_j}^n)(y)h(e_n)|\\
 &+\sum_{j=j_0}^\infty|\sum_{n\in B_j\setminus D_j}(f_n-c_jf_{\al_j}^n)(y)h(e_n)|\\
 &\leq \sum_{j=j_0}^\infty (j+2)\epsilon_j\|y\|+2\sum_{j=j_0}^\infty\epsilon_j\sum_{t<j}m_t \|y\|\\
 &\leq \frac{6}{2^{j_0}}\|y\|
\end{align*}
Estimate the third term as follows
\begin{align*}
 \|\sum_{j=j_0}^\infty\sum_{n\in B_j} c_jf^n_{\al_j}(y)e_n\|&\leq \sum_{j=j_0}^\infty c_j\|\sum_{n\in B_j} f^n_{\al_j}(y)e_n\|\leq \sum_{j=j_0}^\infty\frac{1}{2^{j+1}} \|y\|\leq \frac{1}{2^{j_0}}\|y\|
\end{align*}
using (F7) and the fact that for any block sequence $(h_k)\subset K$ and $z\in X_S$ we have $\|\sum_kh_k(z)e_k\|\leq \|z\|$ (cf. Fact 1.3 \cite{mp2}).
\end{proof}

The following theorem summarizes previous estimates setting the ground for the general criterion (Prop.  \ref{main-criterion}) in the Schlumprecht space. 

\begin{theorem}\label{operator} Take a core tree $\mt$ with parameters $(m_j,q_j)_j$, satisfying (F1)-(F7) and the operator $T$ defined by sequences $(f_n,x_n)_n$ associated to $\mt$. Then we have the following
\begin{enumerate}
    \item the operator $T$  is bounded,  strictly singular and  non-compact.
    \item for any $N\in\N$ and any $q_{j_N-1}\leq k\leq m_{j_N}$ (where $\al_{j_N}$ is minimal in $\{\al\in\mt: |\al|=N\}$ with respect to  $\leq_{lex}$) we have 
$$
c_k(T)\leq 2\sum_{j<j_N}m_j,  \ \ d_k(T)\leq 14k\frac{\log_2(m_{j_{N-1}}+1)}{m_{j_{N-1}}}
$$
\item $T\in\mathscr{S}_\infty(X_S,X_S)$, where $X_S$ is considered with the basis $(e_n)_n$.
\end{enumerate}
\end{theorem}

\begin{proof}
(1) follows immediately from Proposition \ref{est-functionals-2}, as $X_S$ does not contain $c_0$. The non-compactness is guaranteed by the fact that $Tx_n=e_n$, $n\in\N$, with (F1) (recall that $(x_n)_n$ is seminormalized). 

(2)  follows by Propositions \ref{est-vectors} and \ref{est-functionals} and the definition of  $(d_j)_j$. 

(3)  Take a seminormalized sequence $(y_m)_m\subset X_S$ with $\|Ty_l\|_{X_S,\infty}\to 0$, $l\to\infty$. As $\|T(\cdot)\|_{X_S,\infty}\geq \frac{1}{j}\|\cdot\|_{j}$ any $j\in\N$, it follows that for any $j\in\N$ $\|y_l\|_{m_j}\to 0$, $l\to\infty$.

Let $C=\sup_l\|y_l\|$. We show that any subsequence $(y_l)_{l\in M}$ contains a further subsequence $(y_l)_{l\in L}$ with $\|Ty_l\|\xrightarrow{L\ni l\to\infty} 0$, which ends the proof.
Take a subsequence $(y_l)_{l\in M}$. Diagonalizing pass to a subsequence $(y_l)_{l\in L}$, such that $\|y_l\|_{m_j}\leq \frac{1}{2^l}$ for any $j\leq l$, $l\in L$. Then by Prop. \ref{est-functionals-2} we have for any $l\in L$
$$
\|Ty_l\|\leq \sum_{j=1}^l\|y_l\|_{m_j}+\frac{7}{2^l}\|y_l\|\leq \frac{l}{2^l}+\frac{7C}{2^l} \to 0, \ \  L\ni l\to\infty
$$
\end{proof}

\subsection{Small operator ideals}

Fix two core trees: $\mt$ with parameters $(m_j,q_j)_j$  and $\mrr$ with parameters $(k_i,p_i)_i$, with both sets of parameters satisfying (F1)-(F7). Pick block sequences of vectors $(x_n)$ and functionals $(f_n)$ associated to $\mt$ and block sequences of vectors $(y_n)$ and functionals $(g_n)$ associated to $\mrr$ and associated  operators $T,R\in\mathscr{B}_{seq}(X_S)$:
\begin{align}\label{two-operator-def}
T: X_S\ni x\mapsto \sum_n f_n(x)e_n\in X_S, \ \ R: X_S\ni x\mapsto \sum_n g_n(x)e_n\in X_S. \ \
\end{align}

We define now conditions concerning only two fixed consecutive levels of the trees $\mt$ and $\mrr$.

For any $N\in\N$ let $j_N$ (resp. $i_N$) be such that $\al_{j_N}$ is minimal in $\{\al\in\mt: |\al|=N\}$ (resp. $\gamma_{i_N}$ is minimal in $\{\gamma\in\mrr: |\gamma|=N\}$) with respect to the lexicographical order $\leq_{lex}$.
We define the following property: we write that $(m_\al,q_\al)_{|\al|=N-1,N}\gtrdot (k_\gamma,p_\gamma)_{|\gamma|=N-1,N}$, for $0<N\in\N$, provided
\begin{enumerate}
\item[(L1)]
\begin{align*}\label{fast-increase-2}
\frac{\log_2 (m_{j_{N-1}}+1)}{m_{j_{N-1}}}\sum_{|\gamma|< N}k_\gamma\leq \frac{1}{N}
\end{align*}
\item[(L2)] $q_{j_N-1}\leq k_{i_N}\leq m_{j_N}$.
\end{enumerate}

\begin{remark}\label{rem-gtrdot}
Notice that it is possible to build core trees with parameters $(m_j,q_j)_{j=N-1,N}\gtrdot (k_i,p_i)_{i=N-1,N}$ for any $N\in\N$, with both sets of parameters satisfying (F1)-(F7) by induction on $N$. Indeed, choose $m_0,k_0,q_0,p_0$, to ensure (F1), (F2), (L1), (L2). For a fixed $N\in\N$, having chosen parameters of the trees $\mt$, $\mrr$ up to level $N-1$, i.e. $(m_\al)_{|\al|<N}$, $(q_\al)_{|\al|<N}$ (in particular $q_{j_N-1}$, since $|\al_{j_N-1}|=N-1$) and $(k_\gamma)_{|\gamma|<N}$, $(p_\gamma)_{|\gamma|<N}$, we pick $k_{i_N}\geq q_{j_N-1}$ and the rest of parameters on the $N$-th level of the tree $\mrr$, i.e. $(p_\gamma)_{|\gamma|=N}$ and $(k_\gamma)_{|\gamma|=N, \gamma>_{lex}\gamma_{i_N}}$ so that  (F1)-(F7) in the case of $\mrr$ are satisfied. Then we pick $m_{j_N}\geq k_{i_N}$ big enough to ensure (F1)-(F7) in the case of $\mt$ and  so that
$$
\frac{\log_2 (m_{j_N}+1)}{m_{j_N}}\sum_{|\gamma|\leq N}k_\gamma\leq \frac{1}{N+1}
$$
Then we choose $(q_\al)_{|\al|=N}$ and $(m_\al)_{|\al|=N, \al>_{lex} \al_{j_N}}$ to ensure (F1)-(F7) in the case of $\mt$ and thus we finish the inductive procedure.
\end{remark}

\begin{lemma}\label{dyadic}
 There is a family $(T_r)_{r\in\R}$ of bounded strictly singular operators on $X_S$ such that for any $r,s\in\R$, $r\neq s$ we have
 $$
\inf_{k\in\N}\frac{1}{k}c_k(T_r)d_k(T_s)=0 
$$
\end{lemma}
\begin{proof} Let $\mathcal{D}$ be a dyadic tree with the root $\emptyset$, the order $\preceq$, the lexicographic order $\leq_{lex}$ and levels $|\dd|$, $\dd\in\mathcal{D}$. For any $\dd\in\mathcal{D}$, $\dd\neq\emptyset$ let $\dd-$ be the immediate predecessor of $\dd$ in $(\mathcal{D},\preceq)$, i.e. $\dd\in\suc(\dd-)$. Recall that a branch of a tree is a maximal linearly ordered subset of $\mathcal{D}$ with respect to $\preceq$.

We attach to every node $\bf d\in \mathcal{D}$ a set of parameters $\mm^\dd=(m_j^\dd,q_j^\dd)_{j\in A_\dd}$ for some finite $A_\dd\subset\N$ so that
 \begin{enumerate}
  \item for any branch $\mathcal{B}\subset\mathcal{D}$ we have $\N=\bigcup_{\dd\in \mathcal{B}}A_\dd$, and for any $\dd,\dd'\in \mathcal{B}$ with $\dd\preceq \dd'$ we have $\max A_\dd<\min A_{\dd'}$,
  \item for any branch $\mathcal{B}\subset\mathcal{D}$ parameters $(m_j^\dd)_{j\in A_\dd,\dd\in \mathcal{B}}$ define a core tree $\mt_\mathcal{B}$, such that for any $\dd\in \mathcal{B}$ we have $\{j\in\N: |\al_j|=|\dd|\}=A_\dd$ (i.e. $\mm^\dd$ is the set of parameters on the $|\dd|$-level of $\mt_\mathcal{B}$),
  \item for any branch $\mathcal{B}\subset\mathcal{D}$ parameters $\cup_{\dd\in \mathcal{B}}\mm^\dd=(m_j^\dd,q^\dd_j)_{j\in A_\dd,\dd\in \mathcal{B}}$ satisfy (F1)-(F7),
  \item for any $0<N\in\N$, any $\dd,\dd'\in \mathcal{D}$ with $|\dd|=|\dd'|=4N$, $\dd\leq_{lex}\dd'$ we have $\mm^{\dd-}\cup \mm^{\dd}\gtrdot \mm^{\dd'-}\cup \mm^{\dd'}$,
  \item for any $N\in\N$, any $\dd,\dd'\in \mathcal{D}$ with $|\dd|=|\dd'|=4N+2$, $\dd\leq_{lex}\dd'$ we have $\mm^{\dd-}\cup \mm^{\dd}\lessdot \mm^{\dd'-}\cup \mm^{\dd'}$.
 \end{enumerate}
Conditions (1)-(3) above describe the way the parameters defining trees $\mt_\mathcal{B}$, with $\mathcal{B}$ - a branch of $\mathcal{D}$, are represented in the tree $\mathcal{D}$. For any branch $\mathcal{B}$ of $\mathcal{D}$ let $T_\mathcal{B}\in\mathscr{B}_{seq}(X_S)$ be the operator defined by sequences associated to the tree $\mt_\mathcal{B}$ with parameters $\cup_{\dd\in \mathcal{B}}\mm^\dd=(m_j^\dd,q^\dd_j)_{j\in A_\dd,\dd\in \mathcal{B}}$. Conditions (4) and (5) above imply, by Theorem \ref{operator} (2), that for any different branches $\mathcal{B},\mathcal{B}'$ of $\mathcal{D}$ we have
$$
\inf_{k\in\N}\frac{1}{k}c_k(T_\mathcal{B})d_k(T_{\mathcal{B}'})=0 
$$

In order to finish the proof notice that we can choose tuples $\mm^\dd=(m_j^\dd,q_j^\dd)_{j\in A_\dd}$, $\dd\in\mathcal{D}$ satisfying (1)-(5) by induction on the level of $\mathcal{D}$, more precisely choosing for every $0<N\in\N$ parameters $(m_j^\dd,q_j^\dd)_{j\in A_\dd, |\dd|=2N-1,2N}$ satisfying (1)-(3) and either (4) or (5) as in  Remark \ref{rem-gtrdot}, profiting from the fact that on each level of $\mathcal{D}$ there are finitely many $\dd$'s.
\end{proof}

\begin{theorem}\label{main}
 There is a family of small closed operator ideals $(I_A)_{A\subset\R}$ on the Schlumprecht space such that $I_A\subset I_B\Leftrightarrow A\subset B$, for any $A,B\subset \R$. In particular there are exactly $2^\mathfrak{c}$ distinct small closed  ideals on the Schlumprecht space. There is also a chain of cardinality $\mathfrak{c}$ of small closed ideals and an antichain of cardinality $2^\mathfrak{c}$ of small closed ideals on the Schlumprecht space.
 \end{theorem}
 
\begin{proof} Take a family $(T_r)_{r\in\R}\subset \mathscr{S}(X_S)$ as in Lemma \ref{dyadic} and for any $A\subset\R$ define $I_A$ to be the closed operator ideal generated by $(T_r)_{r\in A}$. Then obviously $A\subset B$ implies $I_A\subset I_B$. On the other hand, fix $A\subset\R$ and $r\in \R\setminus A$. Then by Theorem \ref{operator} (3) operators $T_r$ and $(T_s)_{s\in A}$  satisfy the assumptions of Prop. \ref{main-criterion}, thus in particular $T_r\not\in I_A$. In other words $T_r\in I_A$ implies $r\in A$, which finishes the proof.
\end{proof}

\begin{remark}

 Theorem \ref{main} holds true in the case of Banach spaces with a basis whose spreading model is equivalent to the basis $(e_n)_n$ of the Schlumprecht space, like Gowers-Maurey space (cf. \cite{mp2}). 

The last step of construction (creating $(I_A)_{A\subset\R}$ from a family of bounded operators $(T_r)_{r\in\R}$ with certain separation property described in Prop. \ref{main-criterion})
works as in \cite{js}, however, the constructed operators are of a different type.

\end{remark}

\section{Schreier spaces}
\subsection{Small operator ideals}

We recall the definition of  Schreier spaces $X[\mathcal{S}_N]$, $N\in\N$, given in \cite{aa}. 
Define the Schreier families $\mathcal{S}_N$, $N\in\N$, by induction. Let $$
\mathcal{S}_0=\{\{n\}: n\in\N\}\cup\{\emptyset\}
$$
and for any $N\in\N$ let
$$
\mathcal{S}_{N+1}=\left\{ \bigcup_{i=1}^kE_i: k\in\N,  E_1,\dots,E_k\in\mathcal{S}_N, k\leq\min E_1, E_1<\dots<E_k\right\}\cup\{\emptyset\}
$$
We recall here one of the basic properties of Schreier families we shall need later, namely the spreading property: for any $N\in\N$ and  $\{n_1,\dots,n_k\}\in\mathcal{S}_N$, and $m_1<\dots<m_k\in\N$ with $n_1\leq m_1, \dots, n_k\leq m_k$ also $\{m_1,\dots,m_k\}\in\mathcal{S}_N$.

The Schreier space $X[\mathcal{S}_N]$, $N\in\N$, is the completion of $c_{00}$ with the norm
$$
\|(x_n)_n\|_{\mathcal{S}_N}=\sup_{ A\in\mathcal{S}_N, A\neq\emptyset}\sum_{n\in A}|x_n|, \ \ \ (x_n)_n\in c_{00}
$$
The unit vector basis of $c_{00}$, which we denote in this section by $(\tilde{e}_n)_n$, is a shrinking unconditional basis for any $X[\mathcal{S}_N]$, $N\in\N$ (\cite{aa}). For any $I\subset\N$ by $P_I$ denote the canonical projection on $[\tilde{e}_n: n\in I]$.

In order to build a family of $2^\mathfrak{c}$ many distinct operator ideals on the Schreier spaces we shall need a family of $\mathfrak{c}$ many pairwise incomparable subsequences of the basis, i.e. no member of the family dominates the other (in fact we need even stronger condition described in Lemma \ref{sch-dyadic}). 
We shall need the following version of Lemma 3.4 \cite{gl}, providing a necessary condition for a subsequence of the basis to dominate another subsequence of the basis. 

For any infinite $I,J\subset\N$, $I=(i_n)_n$, $J=(j_n)_n$ let $\phi_{I,J}:I\in i_n\mapsto j_n\in J$. 

\begin{lemma}\cite[Lemma 3.4]{gl}\label{sch-rij}
Fix $N\in\N$, $N\geq 1$, and infinite $I,J\subset\N$. 

Assume that for any $k\in\N$ there are  $\mathcal{S}_N$ maximal sets $F_1<\dots<F_k\subset I$   so that $\phi_{I,J}(F_1\cup \dots\cup F_k)\in \mathcal{S}_1$. 

Then the sequence $(\tilde{e}_i)_{i\in I}$ does not dominate the sequence $(\tilde{e}_j)_{j\in J}$. 
\end{lemma}

We define now operators generating distinct operator ideals. Recall  that by \cite{o} any Schreier space $X[\mathcal{S}_N]$, $N\in\N$, is saturated by  $c_0$.  

We fix for the rest of the section a  normalized basic sequence $(e_n)_n$ equivalent to the unit vector basis of $c_0$ and let $Y\subset X[\mathcal{S}_N]$ be the closed subspace spanned by $(e_n)_n$.  Let $T:X[\mathcal{S}_N]\to Y$ be the bounded operator carrying each $\tilde{e}_n$ to $e_n$. By \cite[Lemma 4.14]{bkl} the operator $T$ is strictly singular. 

For any infinite $I\subset \N$ let $T_I=T\circ P_I$. Notice that $T_I\in\mathscr{B}_{seq}(X[\mathcal{S}_N],Y)$ with associated sequences 
$(\tilde{e}_{i}^*,\tilde{e}_{i})_{i\in I}$. 
By construction   $T_I\in\mathscr{S}_\infty(X[\mathcal{S}_N],Y)\cap\mathscr{S}(X[\mathcal{S}_N])$. However,  $\inf_k\frac{1}{k}c_k(T_J)d_k(T_I)=1$ for any $I,J\subset\N$ thus we cannot apply Prop. \ref{main-criterion}. Due to results of \cite{gl} we are able to prove the following, somewhat analogous, result. 

\begin{proposition}\label{sch-main-criterion}  Fix $N\in\N$, $N\geq 1$, infinite $I\subset \N$ and a family $\mathcal{J}$ of infinite subsets of $\N$. Assume that for any $J_1,\dots, J_k\in \mathcal{J}$, $k\in\N$,   the sequence $(\tilde{e}_i)_{i\in I}$ does not dominate  $(\tilde{e}_j)_{j\in J}$, where $J=J_1\cup\dots\cup J_k$.

Then $\|T_I-Q\|\geq 1$ for any $Q$ in the operator ideal generated by the family of operators $\{T_J: J\in \mathcal{J}\}$. 
\end{proposition}

\begin{proof} The reasoning follows the first part of the proof of Prop. \ref{main-criterion}, using results of \cite{gl}.

Take $I$ and $\mathcal{J}$  as in the theorem. Assume there is $Q$ in the operator ideal generated by $\{T_J: J\in \mathcal{J}\}$ with $\|T_I-Q\|<1$. Then there are $Q_t,S_t\in\mathscr{B}(X)$ and $J_t\in\mathcal{J}$, $t=1,\dots,k$,  such that
$\|T_I-\sum_{t=1}^{k}Q_t\circ T_{J_t}\circ S_t\|< 1$. As $T_I(\tilde{e}_i)=e_i$, $i\in I$, the sequence $((\sum_{t=1}^kQ_t\circ T_{J_t}\circ S_t)\tilde{e}_i)_{i\in I}$ is seminormalized. It follows that for any $i\in I$ there is $t_i\in\{1,\dots,k\}$ so that the sequence $((T_{J_{t_i}}\circ S_{t_i})\tilde{e}_i)_{i\in I}$ is seminormalized. For any $t=1,\dots,k$ let $I_t=\{i\in I: t=t_i\}$ and note that $(I_t)_{t=1,\dots,k}$ form a partition of $I$. Let $J=J_1\cup\dots\cup J_k$. 

Define a bounded operator $R=\sum_{t=1}^k P_J\circ S_t\circ P_{I_t}: [\tilde{e}_i: i\in I]\to [\tilde{e}_j: j\in J]$ and notice that for any $i\in I$ we have 
\begin{align*}
\sup_{j\in J}|(R\tilde{e}_i)(j)|&=\sup_{j\in J} |(S_{t_i}\tilde{e}_i))(j)|\geq \sup_{j\in J_{t_i}} |(S_{t_i}\tilde{e}_i))(j)|\\
&=
\sup_{j\in J_{t_i}}|e^*_j(( T_{J_{t_i}}\circ S_{t_i})(\tilde{e}_i))|\geq c\|(T_{J_{t_i}}\circ S_{t_i})\tilde{e}_i\|_{\mathcal{S}_N}    \end{align*}
where $c>0$ is the equivalence constant of $(e_n)_n$ to the unit vector basis of $c_{00}$.
As the sequence $((T_{J_{t_i}}\circ S_{t_i})\tilde{e}_i)_{i\in I}$ is seminormalized, by 
\cite[Thm 1.1]{gl} the sequence $(\tilde{e}_i)_{i\in I}$ dominates $(\tilde{e}_j)_{j
\in J}$, which contradicts the assumptions and ends the proof. 
\end{proof}

\begin{lemma}\label{sch-dyadic}
There is a family $(J_r)_{r\in\R}$ of infinite subsets of $\N$ such that for any pairwise different $r, r_1, \dots, r_k\in\R$, $k\in\N$,  the sequence $(\tilde{e}_i)_{i\in J_r}$ does not dominate the sequence $(\tilde{e}_j)_{j\in J_{r_1}\cup\dots \cup J_{r_k}}$.  
\end{lemma}
\begin{proof}
Let $\mathcal{D}$ be a dyadic tree with the root $\emptyset$, the order $\preceq$, the lexicographic order $\leq_{lex}$ and levels $|\dd|$, $\dd\in\mathcal{D}$. For any $\dd\in\mathcal{D}$, let $\dd^+$ be the immediate successor of $\dd$ in $(\mathcal{D},\leq_{lex})$.  For any $\dd\in\mathcal{D}$ define inductively on $(\mathcal{D}, \leq_{lex})$ a finite set $F_\dd\subset\N$ so that for any $\dd\in\mathcal{D}\setminus\emptyset$ we have
\begin{enumerate}
    \item $F_{\dd^+}>F_\dd$,
    \item $F_\dd=H_\dd\cup G_\dd$, $H_\dd<G_\dd$,
    \item $\# H_\dd=\sum_{\dd'<_{lex}\dd}\# F_{\dd'}$,
    \item $G_\dd=E_1^\dd\cup\dots\cup E_{|\dd|}^\dd$, where   $E_1^\dd<\dots<E_{|\dd|}^\dd$ are maximal $\mathcal{S}_N$ sets.
\end{enumerate}
Given any $\mathcal{A}\subset\mathcal{D}$ let $J_\mathcal{A}=\cup_{\dd\in\mathcal{A}}F_\dd$. Notice that by (3) for any branch $\mathcal{B}$ of $\mathcal{D}$,  any $\dd\in\mathcal{B}$  and any infinite $\mathcal{A}\subset \mathcal{D}\setminus \{\dd\}$ we have $\phi_{J_\mathcal{B}, J_\mathcal{A}}(F_\dd)>\max F_\dd$, thus $\phi_{J_\mathcal{B}, J_\mathcal{A}}(G_\dd)$ is an $\mathcal{S}_1$ set. By Lemma \ref{sch-rij} for any branch $\mathcal{B}$ of $\mathcal{D}$ and any infinite $\mathcal{A}\subset\mathcal{D}$ with $\mathcal{B}\setminus\mathcal{A}$ infinite, the sequence $(\tilde{e}_i)_{i\in J_\mathcal{B}}$ does not dominate the sequence $(\tilde{e}_j)_{j\in J_\mathcal{A}}$. In particular the family  $\{J_\mathcal{B}: \ \mathcal{B}\ - \text{ a branch of }\mathcal{D}\}$ satisfy the assertion of the Lemma.

\end{proof}

\begin{theorem}\label{main-schreier} Fix $N\in\N$, $N\geq 1$. There is a family of small closed operator ideals $(I_A)_{A\subset\R}$ on the Schreier space $X[\mathcal{S}_N]$ such that $I_A\subset I_B\Leftrightarrow A\subset B$, for any $A,B\subset \R$. In particular there are exactly $2^\mathfrak{c}$ distinct small closed  ideals on the Schlumprecht space. There is also a chain of cardinality $\mathfrak{c}$ of small closed ideals and an antichain of cardinality $2^\mathfrak{c}$ of small closed ideals on the Schreier space $X[\mathcal{S}_N]$.
 \end{theorem}

\begin{proof} Take a family $(T_{J_r})_{r\in\R}\subset \mathscr{S}(X_S)$ with $(J_r)_{r\in\R}$ as in Lemma \ref{sch-dyadic} and for any $A\subset\R$ define $I_A$ to be the closed operator ideal generated by $(T_{J_r})_{r\in A}$. Then obviously $A\subset B$ implies $I_A\subset I_B$. On the other hand, fix $A\subset\R$ and $r\in \R\setminus A$. Then by the choice of $(J_r)_{r\in \R}$ sets  $J_r$ and $(J_s)_{s\in A}$  satisfy the assumptions of Prop. \ref{sch-main-criterion}, thus in particular $T_{J_r}\not\in I_A$. In other words $T_{J_r}\in I_A$ implies $r\in A$, which finishes the proof.
\end{proof}

\subsection{Operators factoring through $c_0$}
As $X[\mathcal{S}_N]$ contains a copy of $c_0$, by Sobczyk theorem the ideal $\overline{\mathscr{G}}_{c_0}(X[\mathcal{S}_N])$, i.e. the closure of the ideal of operators factoring through $c_0$ is not contained in $\mathscr{S}(X[\mathcal{S}_N])$. We show below that in the case $N\geq 2$ the reverse inclusion does not hold either, answering one of the questions of \cite{bkl}. The case $N=1$ remains open. 

The following lemma is proved in \cite[Prop. 0.7]{cs} in the case of $N=1$ (recall that $X[\mathcal{S}_0]=c_0$).

\begin{lemma}\label{sch-av} For any $N\in\N$, $N\geq 1$, and any sequence $(E_m)_{m\in \N}$ of finite subsets of $\N$ with $\# E_m\leq\min E_m$, $\# E_{m+1}>2\# E_m\max E_m$, $m\in\N$ the sequence of averages $(y_m)_m\subset X[\mathcal{S}_N]$ defined as $y_m=\frac{1}{\# E_m}\sum_{n\in E_m}\tilde{e}_n$, $m\in\N$, is 2-equivalent to the subsequence
$(\tilde{e}_{s_m})_m$ of the basis of $X[\mathcal{S}_{N-1}]$, where $s_m=\max E_m$, $m\in\N$.
\end{lemma}

\begin{proof} We prove first that $(y_m)_m\subset X[\mathcal{S}_N]$ is 2-dominated by $(\tilde{e}_{s_m})_m\subset X[\mathcal{S}_{N-1}]$. 

Take a finite sequence of  scalars $(a_m)_m$ and a maximal $\mathcal{S}_N$ set $F$  with $\|\sum_ma_my_m\|_{\mathcal{S}_N}=\|\sum_ma_my_m|_F\|_{\ell_1}$. By \cite[Lemma 3.8]{gl} $F=\cup_{i=1}^kF_i$, where $F_1<\dots<F_k$ are maximal members of $\mathcal{S}_1$ and $(\min F_i)_{i=1}^k\in\mathcal{S}_{N-1}$. for each $i=1,\dots,k$ let $G_i=\{m\in\N: F_i\cap E_m\neq\emptyset\}$ and pick $m_i\in G_i$ with $|a_{m_i}|=\max\{|a_m|: m\in G_i\}$. Estimate, using the fact that $(y_m)_m$ in $X[\mathcal{S}_1]$ is 1-equivalent to the unit vector basis of $c_0$ (\cite[Prop. 0.7]{cs}),
$$
\|\sum_ma_my_m|_F\|_{\ell_1}=
\sum_{i=1}^k\|\sum_{m\in G_i}a_my_m|_{F_i}\|_{\ell_1}\leq \sum_{i=1}^k\|\sum_{m\in G_i}a_my_m\|_{\mathcal{S}_1}\leq
\sum_{i=1}^k
|a_{m_i}|\leq
\dots
$$
Notice that for any $m\in\N$, as each $F_i$ is a maximal $\mathcal{S}_1$ set, there can be at most 2 sets from the family $(F_i)_{i=1}^k$ intersecting $E_m$, thus each $m_i$ can appear in the above sum at most twice. Thus, as $s_{m_i}\geq \min F_i$, $i=1,\dots, k$, and $(\min F_i)_{i=1}^k\in\mathcal{S}_{N-1}$, 
we conclude
$$
\dots\leq 2\|\sum_m a_{m}\tilde{e}_{s_m}\|_{\mathcal{S}_{N-1}}
$$

Now we prove that $(\tilde{e}_{s_m})_m\subset X[\mathcal{S}_{N-1}]$ is 2-dominated by $(y_m)_m\subset X[\mathcal{S}_N]$. Take a finite sequence of scalars $(a_m)_m$ and  a set $F\in\mathcal{S}_{N-1}$ such that $\|\sum_ma_m\tilde{e}_{s_m}\|_{\mathcal{S}_{N-1}}=\|\sum_{s_m\in F}a_m\tilde{e}_{s_m}\|_{\ell_1}$. As $s_m\leq E_{m+1}$ for each $m$, the set $G=\bigcup_{s_m\in F}E_m$ is a union of two $\mathcal{S}_N$ sets, thus 
$$
2\|\sum_ma_my_m\|_{\mathcal{S}_N}\geq \sum_{s_m\in F}|a_m|\sum_{n\in E_m}\frac{1}{\# E_m}=\sum_{s_m\in F}|a_m|
$$
which ends the proof.
\end{proof}

The following observation extends \cite[Lemma 4.14]{bkl}.

\begin{proposition} The formal identity operator $R_N:X[\mathcal{S}_N]\to X[\mathcal{S}_{N-1}]$, $N\in\N, N\geq 1$ is strictly singular.

Moreover, if $N\geq 2$, then any bounded operators $T:X[\mathcal{S}_N]\to c_0$ and $S:c_0\to X[\mathcal{S}_{N-1}]$ satisfy
$\|R_N-S\circ T\|\geq 1$.

\end{proposition}

\begin{proof} We prove strict singularity by induction on $N$. The case $N=1$ follows by \cite[Lemma 4.14]{bkl}. Let $N\geq 2$ and assume $R_{N-1}$ is strictly singular, whereas   $R_N$ is not strictly singular. Then on some block subspace $Z\subset X[\mathcal{S}_N]$ norms $\|\cdot\|_{\mathcal{S}_N}$ and $\|\cdot\|_{\mathcal{S}_{N-1}}$ are equivalent. As $R_{N-1}$ is strictly singular there is a block sequence $(z_n)_n\subset Z$ with $\|z_n\|_{\mathcal{S}_{N-1}}=1$ and $\|z_n\|_{\mathcal{S}_{N-2}}\leq \frac{1}{2^n\max\supp z_{n-1}}$. Fix $K\in\N$ and Let $z=\sum_{n=K+1}^{2K}z_n$. Note that $\|z\|_{\mathcal{S}_N}\geq K$. 

In order to estimate $\|z\|_{\mathcal{S}_{N-1}}$ pick $F\in\mathcal{S}_{N-1}$. Then $F=\sum_{l=1}^LF_l$ for some $F_1<\dots<F_L\in\mathcal{S}_{N-2}$ and $L\leq \min F_1$. Let $n_0=\min\{n\in\N:z_n|_F\neq 0\}$, obviously  $L\leq \max\supp z_{n_0}$. Estimate
\begin{align*}
\|z|_F\|_{\ell_1}&\leq\|z_{n_0}|_F\|_{\ell_1}+ \sum_{n=n_0+1}^K\sum_{l=1}^L\|z_n|_{F_l}\|_{\ell_1}\\
&\leq 1+\sum_{n=n_0+1}^KL\|z_n\|_{\mathcal{S}_{N-2}}\\
&\leq 1+\sum_{n>n_0}\frac{1}{2^n}\leq 2
\end{align*}
As $K\in\N$ is arbitrary this yields contradiction and ends the proof of strict singularity of $R_N$.     

Now let $N\geq 2$. Assume for some bounded operators 
$T:X[\mathcal{S}_N]\to c_0$ and $S:c_0\to X[\mathcal{S}_{N-1}]$ we have
$\|R_N-S\circ T\|< 1$. As $R_N(\tilde{e}_n)(n)=1$, $n\in\N$, it follows that 
\begin{align*}
\delta:=\inf_n(S\circ T)\tilde{e}_n)(n)>0
\end{align*}

As $(\tilde{e}_n)_n$ is weakly null, $(T\tilde{e}_n)_n$ and $((S\circ T)\tilde{e}_n)_n$ are also weakly null, thus we can choose some infinite $M\subset\N$ such that  the sequence $(T\tilde{e}_n)_{n\in M}$ is equivalent to the unit vector basis of $c_0$, whereas $|((S\circ T)\tilde{e}_n)(m)|<\frac{1}{2^{n+m+1}}$ for any $n,m\in M$, $n\neq m$. 

Let $P_M:X[\mathcal{S}_{N-1}]\to [\tilde{e}_n: n\in M]\subset X[\mathcal{S}_{N-1}]$ be the canonical basis projection. By the choice of $M$ we have $(P_M\circ S)(T\tilde{e}_n)=\alpha_n\tilde{e}_n+w_n$ with $\alpha_n\geq \delta$ and $\|w_n\|_{\mathcal{S}_{N-1}}\leq\|w_n\|_{\ell_1}\leq \frac{1}{2^n}$ for any $n\in M$. Therefore the sequence $(\tilde{e}_n)_{n\in M}$ in $X[\mathcal{S}_{N-1}]$, as equivalent to $((P_M\circ S)(T\tilde{e}_n))_{n\in M}$, is dominated by $(T\tilde{e}_n)_{n\in M}$. Thus we obtain a contradiction, as $(T\tilde{e}_n)_{n\in M}$ is equivalent to the unit vector basis of $c_0$. 

\end{proof}

\begin{corollary}
$\mathscr{S}(X[\mathcal{S}_N])\not\subset\overline{\mathscr{G}}_{c_0}(X[\mathcal{S}_N])$ for any  $N\in\N$, $N\geq 2$.
\end{corollary}

\end{document}